\documentclass[11pt,reqno]{amsart}
\usepackage[a4paper,hscale=0.7,vscale=0.7,centering]{geometry}
\usepackage{amssymb}
\usepackage{color}

\newcommand{\erre}{\mathbb{R}}
\newcommand{\enne}{\mathbb{N}}
\newcommand{\C}{\mathbb{C}}
\newcommand{\D}{\mathbb{D}}
\newcommand{\E}{\mathbb{E}}
\newcommand{\LL}{\mathbb{L}}
\renewcommand{\P}{\mathbb{P}}

\newcommand{\ip}[2]{\langle #1,#2 \rangle}
\newcommand{\bip}[2]{\left\langle #1,#2 \right\rangle}
\newcommand{\Cpol}{C_{\mathrm{pol}}}
\newcommand{\co}{\mathcal{O}}

\newtheorem{prop}{Proposition}[section]
\newtheorem{thm}[prop]{Theorem}

\newtheorem{lemma}[prop]{Lemma}
\newtheorem{defi}[prop]{Definition}
\theoremstyle{definition}
\newtheorem{hyp}[prop]{Hypothesis}
\theoremstyle{remark}
\newtheorem{rmk}[prop]{Remark}
\newtheorem{example}[prop]{Example}

\numberwithin{equation}{section}

\begin{document}

\title[Densities for parabolic semilinear SPDEs]{Existence and
regularity of the density for the solution to
semilinear dissipative parabolic SPDEs}

\author{Carlo Marinelli}
\address[Carlo Marinelli]{Institut f\"ur Angewandte Mathematik,
  Universit\"at Bonn, Germany, and Facolt\`a di Economia, Libera
  Universit\`a di Bolzano, Italy}
\urladdr{http://www.uni-bonn.de/$\sim$cm788}

\author{Eulalia Nualart}
\address[Eulalia Nualart]{Institut Galil\'ee, Universit\'e Paris 13,
  93430 Villetaneuse, France}
\email{eulalia@nualart.es}
\urladdr{http://nualart.es}

\author{Llu\'is Quer-Sardanyons}
\address[Llu\'is Quer-Sardanyons]{Departament de Matem\`atiques,
  Universitat Aut\`onoma de Barcelona, 08193 Bellaterra (Barcelona),
  Spain. Tel. 0034935814542, Fax 0034935812790.}
\email{quer@mat.uab.cat}
\urladdr{http://mat.uab.cat/~quer}

\begin{abstract}
  We prove existence and smoothness of the density of the solution to
  a nonlinear stochastic heat equation on $L^2(\mathcal{O})$ (evaluated at
  fixed points in time and space), where $\mathcal{O}$ is an open bounded domain in
  $\mathbb{R}^d$ with smooth boundary. The equation is driven by an
  additive Wiener noise and the nonlinear drift term is the
  superposition operator associated to a real function which is
  assumed to be (maximal) monotone, continuously differentiable, and
  growing not faster than a polynomial. The proof uses tools of the
  Malliavin calculus combined with methods coming from the theory of
  maximal monotone operators.
\end{abstract}

\subjclass[2000]{60H07, 60H15}

\keywords{Stochastic partial differential equation, existence and
  regularity of densities, Malliavin calculus.}

\date{21 February 2012}

\maketitle

%-----------------------------------------------------

\section{Introduction}
Let $\co \subset \erre^d$ be an open bounded domain, and consider the
stochastic semilinear parabolic evolution equation on $L^2(\co)$
\begin{equation}     \label{eq:1}
du(t) - \Delta u(t)\,dt + f(u(t))\,dt =\eta u(t)dt+ B\,dW(t), \qquad u(0)=u_0,
\end{equation}
where $\Delta$ is the Dirichlet Laplacian on $L^2(\co)$, $W$ is a
standard cylindrical Wiener process on $L^2(\co)$, $f:\erre\to\erre$
is continuous, increasing and of polynomial growth, $\eta$ is a
positive number, $B$ is a deterministic bounded linear operator on
$L^2(\co)$, and $u_0$ is an $L^2(\co)$-valued random variable (precise
assumptions are given in the next section). As explained in Remark
\ref{rmk:lip}, the above equation covers the case where $f-\eta$ is
replaced by $f+g$, where $f$ is as above and $g$ is a globally
Lipschitz function.

Under regularity assumptions on the initial datum and the
coefficients, one can prove, using monotonicity methods, that
\eqref{eq:1} admits a unique mild solution with continuous paths (see
e.g. \cite{cerrai-libro} and references therein). The purpose of this
paper is to study the existence and regularity of the density (with
respect to Lebesgue measure) of the random variable $u(t,x)$, with
$(t,x) \in ]0,T] \times \co$. In particular, we show that a sufficient
condition for the existence of the density is that $f$ is of class
$C^1$ with polynomially bounded derivative. We do not claim that this
condition is sharp (as a matter of fact, we conjecture that it is not,
by far), but we do hope that our methods could be the starting point
for further developments. The proof relies on techniques of the
Malliavin calculus and on a priori estimates for solutions of
approximating equations, obtained replacing the nonlinear drift term
by its Yosida regularization. Furthermore, we show that if $f$ is of
class $C^m$ with polynomially bounded derivatives, the density becomes
smoother, as it is natural to expect. The proof of this still relies
on a priori estimates, and requires a further regularization of the
drift, as the Yosida approximation does not necessarily have bounded
derivatives.

There is a large literature on problems of existence and regularity of
densities for solutions to parabolic SPDEs with Lipschitz
non-linearities by means of the Malliavin calculus (see
\cite{Marquez-Sarra,Millet-Sanz-AP1999,nualart-LNM,Nualart-QuerPOTA,Sanz-book}
and references therein). On the other hand, much less attention has
been dedicated to SPDEs with non-Lipschitz coefficients: the first
article (relying on Malliavin calculus) we are aware of is the article
\cite{pardoux-zhang}, where a nonlinear one-dimensional stochastic
heat equation on $[0,1]$ with polynomially growing drift and diffusion
coefficients is considered. Using techniques of the Malliavin
calculus, the authors prove, under quite general conditions, including
a very weak non-degeneracy assumption on the diffusion coefficient,
the absolute continuity of the law of the solution evaluated at fixed
points in time and space. We note that existence and uniqueness of
solution for the stochastic heat equation on $[0,1]$ driven by an
additive noise with a measurable drift having polynomial growth has
been proved in \cite{gyongy-pardoux2}.  In \cite{Fournier}, the
absolute continuity for the law of the solution to a one-dimensional
stochastic heat equation with H\"older continuous diffusion
coefficient and linearly growing drift is proved with completely
different methods. In particular, the proofs involve the Euler
approximation and techniques of harmonic analysis.

As far as applications of Malliavin calculus to non-linear parabolic
SPDEs is concerned, we should mention that the existence of the
density for the stochastic Cahn-Hilliard equation with non-Lipschitz
coefficients has been investigated in
\cite{Cardon-weber}. Furthermore, the existence and smoothness of the
density for the solution to the stochastic Burgers equation with
globally Lipschitz coefficients have been considered in
\cite{Leon-Nualart-Pet,Zaidi-Nualart}.

The paper is organized as follows. In Section \ref{sec:wpa}, we
provide the definition of mild solution to our equation (\ref{eq:1}),
state a well-posedness result for this equation and summarize some
properties of a regularized version of (\ref{eq:1}) in terms of the
Yosida approximations. Section \ref{sec:rf} is devoted to establish
the random field counterpart (i.e. {\it \`a la} Walsh \cite{Walsh}) of
our equation.  In Section \ref{sec:aux}, we collect some auxiliary
results needed for the main theorem's proof. These correspond to a
version of the chain rule for Malliavin derivatives, some properties
of time-dependent evolution operators and estimates for some
regularizations of the drift coefficient.  Eventually, in Section
\ref{sec:exis-smooth}, we state and prove the main result of the
paper, Theorem \ref{thm:main}. For this, we first deal with the
Malliavin differentiability of the solution to our equation and,
secondly, we study the invertibility of the corresponding Malliavin
matrix.

\subsection*{Notation} We shall write $a \lesssim b$ to mean that there
exists a constant $N$ such that $a \leq Nb$. To emphasize that the
constant $N$ depends on the parameters $p_1,\ldots,p_m$, we shall also
write $a \lesssim_{p_1,\ldots,p_m} b$.
We shall write $\sup$ to denote both the supremum and the essential
supremum.
For a Banach space $E$, we shall denote by $\LL^p(E)$ the space of
$E$-valued random variables $\xi$ such that $\E\|\xi\|_E^p < \infty$.

%---------------------------------------------------------------------

\section{Well-posedness and approximation}
\label{sec:wpa}
Let $\co$ be an open bounded domain in $\erre^d$.  $L^p$ spaces over the
domain $\co$ will be denoted without explicitly mentioning the domain,
e.g. $L^2$ stands for $L^2(\co)$. The norm in $L^p$ will be denoted by
$\|\cdot\|_p$, and the inner product of $L^2$ by
$\ip{\cdot}{\cdot}_2$, unless otherwise stated.

On a given stochastic basis $(\Omega,\mathcal{F},\mathbb{F}=
(\mathcal{F}_t)_{0 \leq t \leq T},\P)$, with $T$ a fixed positive
number, let us be given the following semilinear SPDE on $L^2$:
\begin{equation}
  \label{eq:30}
  du(t) - \Delta u(t)\,dt + f(u(t))\,dt = \eta u(t)\,dt + B\,dW(t),
  \qquad u(0)=u_0,
\end{equation}
where $\Delta$ is the Dirichlet Laplacian on $L^2$, $f:\erre \to
\erre$ is continuous, increasing, and such that $|f(x)| \lesssim
1+|x|^p$ for some $p>0$ (we denote the evaluation operator associated
to the function $f$ by the same symbol), $\eta$ is a positive number,
$W$ is a cylindrical Wiener process on $L^2$ generating the filtration
$\mathbb{F}$, $B:L^2\rightarrow L^2$ is a linear and bounded operator,
and $u_0$ is an $\mathcal{F}_0$-measurable $L^2$-valued random
variable such that $\E\|u_0\|_2^2<\infty$. All expressions involving
random quantities are meant to hold $\P$-a.s. if not otherwise
specified.

\begin{rmk}\label{rmk:lip}
  Instead of (\ref{eq:30}) one may equivalently consider an SPDE of
  the type
  \[
  du(t) - \Delta u(t)\,dt + f(u(t))\,dt = B\,dW(t),
  \qquad u(0)=u_0,
  \]
  where $f$ is quasi-monotone, i.e. such that $f+\eta$ is monotone for
  some $\eta>0$. In fact, we may write $f=(f+\eta)-\eta =
  \tilde{f}-\eta$, with $\tilde{f}$ monotone, thus obtaining
  (\ref{eq:30}) with $\tilde{f}$ replacing $f$. Furthermore, note that
  equations of the type
  \[
  du(t) - \Delta u(t)\,dt + f(u(t))\,dt + g(u(t))\,dt = B\,dW(t),
  \qquad u(0)=u_0,
  \]
  with $f$ monotone (or quasi-monotone) and $g$ (globally) Lipschitz
  are just particular cases of (\ref{eq:30}). In fact, one has
  \[
  \big(f(x)+g(x)-f(y)-g(y)\big)(x-y) \geq \big(g(x)-g(y)\big)(x-y)
  \geq -\|g\|_{\dot{C}^{0,1}} |x-y|^2,
  \]
  i.e. $f+g$ is quasi-monotone, since $f+g+\eta$ is monotone, with
  $\eta=\|g\|_{\dot{C}^{0,1}}$. (Here $\|\cdot\|_{\dot{C}^{0,1}}$
  stands for the Lipschitz norm).
\end{rmk}

We shall work with the so-called mild solution, whose definition we
recall. In the following we shall denote by $\Delta$ the realization
of the Dirichlet Laplacian on different function spaces. The same
convention we are going to use for the semigroup $S(t):=e^{t\Delta}$,
$t \geq 0$, generated by $\Delta$.
\begin{defi}
  An $L^2$-valued adapted process $u$ is a mild solution to
  equation \eqref{eq:30} if $f(u) \in L^1([0,T] \to L^2)$
  and it satisfies the integral equation
  \begin{equation} 
    \label{eq:mild} 
    u(t) + \int_0^t S(t-s) \big(f(u(s))-\eta u(s)\big)\,ds =
    S(t)u_0 + \int_0^t S(t-s)B\,dW(s)
  \end{equation}
  for all $t \in [0,T]$.
\end{defi}

As is well-known, in order for the stochastic integral in
\eqref{eq:mild} to be a well-defined (Gaussian) $L^2$-valued random
variable it is necessary (and sufficient) to assume that
\begin{equation}     \label{eq:hsc}
  \int_0^t \operatorname{Tr} \big( S(s) BB^* S(s)^* \big)\,ds < \infty
  \qquad \forall t \in [0,T].
\end{equation}
Let us also recall that this condition is weaker than requiring
$Q:=BB^*$ to be trace-class (cf. e.g. \cite[Ch.~2]{DP-K}).
We shall actually work under the following stronger standing
assumption, which guarantees, as we are going to recall, that
\eqref{eq:30} admits a unique mild solution with paths in a space of
continuous functions.
\begin{hyp}     \label{hyp:csc}
  It holds
  \[
  \E\sup_{t \leq T} \Big\| \int_0^t S(t-s)B\,dW(s) \Big\|^q_{C(\overline{\co})}
  < \infty \qquad \forall q \geq 1.
  \]
\end{hyp}
Conditions on $B$ and $\co$ implying that this hypothesis is fulfilled
are extensively discussed in the literature
(cf. e.g. \cite[Ch.~6]{cerrai-libro} and \cite[Ch.~2]{DP-K}).

\smallskip

In order to state the well-posedness and approximation results we
need, let us fix some further notation: we denote by $\C_q$, $1 \leq q
< \infty$, the space of adapted processes $u$ with values in
$C(\overline{\co})$ such that
\[
\| u \|_{\C_q} := 
\Big( \E \sup_{t\leq T} \|u(t)\|^q_{C(\overline{\co})} \Big)^{1/q}
< \infty.
\]
The following global well-posedness result holds true (see
e.g. \cite[Thm. 7.13 and Rmk. 11.23]{DZ92}, as well as
\cite[Prop.~6.2.2]{cerrai-libro}, for a proof). 
\begin{thm}     \label{thm:wp}
  Let $q \geq 1$ and assume that $u_0 \in \LL^q(C(\overline{\co}))$.
  Then equation \eqref{eq:30} admits a unique mild solution $u \in
  \C_q$.
\end{thm}
Let us now introduce the regularized SPDE
\begin{equation}
  \label{eq:reg}
  du_\lambda(t) - \Delta u_\lambda(t)\,dt + f_\lambda(u_\lambda(t))\,dt 
  = \eta u_\lambda(t)\,dt + B\,dW(t),
  \qquad u(0)=u_0,
\end{equation}
where $f_\lambda$, $\lambda>0$, stands for the Yosida approximation of
$f$, i.e.
\[
f_\lambda(x) := \frac{1}{\lambda} \big(x-J_\lambda(x)\big),
\qquad J_{\lambda}(x) := (I+\lambda f)^{-1}(x).
\]
Recall that $f_\lambda$ is increasing, Lipschitz continuous with
Lipschitz constant bounded by $1/\lambda$, and $f_\lambda \to f$
pointwise as $\lambda \to 0$. Moreover, $J_\lambda$ is a contraction
and $f_\lambda=f \circ J_\lambda$ (see e.g. \cite{Bmax} for a detailed
discussion of the Yosida approximation).  It is then clear that
Theorem \ref{thm:wp} applies also to (\ref{eq:reg}), yielding the
existence and uniqueness of a mild solution $u_\lambda \in
\C_q$. Furthermore, the following convergence result holds true (see
\cite[Prop.~6.2.5]{cerrai-libro}).
\begin{prop}\label{prop:1}
  Let $q \geq 1$ and assume that $u_0 \in \LL^q(C(\overline{\co}))$.
  Then there exists a constant $N$, independent of $\lambda$, such that
  \[
  \E \sup_{t\leq T} \|u_\lambda(t)\|^q_{C(\overline{\co})} < N.
  \]
  Moreover, one has
  \[
  \lim_{\lambda\to 0}
  \E\sup_{t\leq T} \| u_\lambda(t) - u(t) \|^q_{C(\overline{\co})} = 0,
  \]
  where $u \in \C_q$ denotes the (unique) mild solution to equation
  \eqref{eq:30}. In particular, one has, for any $(t,x)\in [0,T]\times
  \overline{\co}$,
  \[
  \lim_{\lambda \to 0} \E |u_\lambda(t,x)-u(t,x)|^q =0.
  \]
\end{prop}

\section{An alternative expression for the SPDE \eqref{eq:30}}
\label{sec:rf}
The pathwise continuity in $t$ and $x$ of the solution to equation
(\ref{eq:mild}) guaranteed by Theorem \ref{thm:wp} allows us to pass to the
random field formulation of (\ref{eq:mild}), i.e. as in \cite{Walsh}, in
a sense made precise in Proposition \ref{prop:rf} below.

It is classical that $S(t)$ is a kernel operator for all $t>0$, i.e.
there exists a function $]0,T] \times \co \times \co \ni (t,x,y)
\mapsto G_t(x,y)$, with $G_t(\cdot,\cdot) \in L^\infty(\co \times
\co)$ for all $t \in ]0,T]$, such that, for any $0 < t \leq T$,
\[
S(t)\phi = \int_\co G_t(\cdot,y)\phi(y)\,dy
\]
and
\[
\| S(t) \|_{1 \to \infty} = \| G_t(\cdot,\cdot) \|_{L^\infty(\co \times \co)},
\]
where $\|\cdot\|_{1 \to \infty}$ stands for the $L^1 \to L^\infty$
operator norm. Since the semigroup $S$ is contracting in $L^\infty$ 
i.e.  $\|S(t)f\|_{L^\infty} \leq \|f\|_{L^\infty}$ for all $t>0$, and
it holds
\[
G_{t}(x,\co) = \int_\co G_{t}(x,y)\,dy = [S(t)1_\co](x),
\]
one has
\begin{equation}    \label{eq:101}
  \sup_{x\in\co} G_{t}(x,\co) = \| S(t)1_\co \|_{L^\infty} \leq 1.
\end{equation}
We have the following result, where we use the terminology introduced
in \cite{Walsh} for stochastic integrals.
\begin{prop}
  \label{prop:rf}
  For $q \geq 1$, let $u_0 \in \LL^q(C(\co))$ and denote the unique
  mild solution in $\C_q$ to equation \eqref{eq:mild} by $u$. Setting
  $u(t,x):=[u(t)](x)$, one has, for any $(t,x) \in ]0,T] \times \co$, 
  \begin{align*}
    u(t,x) &=  \int_\co G_t(x,y) u_0(y)\,dy
    + \int_0^t\!\!\int_\co G_{t-s}(x,y)
      \big(\eta u(s,y) - f(u(s,y))\big)\,dy\,ds\\
    &\quad + \int_0^t\!\!\int_\co G_{t-s}(x,y)\,\bar{W}(ds,dy),
  \end{align*}
  where $\bar{W}$ stands for a martingale measure with covariance
  operator $Q=BB^*$. Moreover,
\begin{equation} \label{unif-sup}
\E \left( \sup_{(t,x) \in \co_T} |u(t,x)|^q \right) < \infty,
\end{equation}
where $\co_T:=[0,T] \times \co$.
\end{prop}
\begin{proof}
  It is readily seen that we only have to prove that, for any $t \in ]0,T]$,
  \begin{equation}
  \int_0^t S(t-s)B\,dW(s) = \int_0^t\!\!\int_\co G_{t-s}(\cdot,y)\,
  \bar{W}(ds,dy)
  \label{eq:422}
  \end{equation}
  as $L^2$-valued random variables. Let us assume
  (without loss of generality) that, formally,
  \begin{equation}
  W(t) = \sum_{k\in\enne} e^k \, w_k(t),
  \label{eq:423}
  \end{equation}
  where $\{e^k\}_{k\in\enne}$ is a basis of $L^2$ and
  $\{w_k\}_{k\in\enne}$ is a family of independent standard
  one-dimensional Brownian motions. Then we have, by the integral
  representation of $S(\cdot)$,
  \begin{align*}
    \int_0^t S(t-s)B\,dW(s) &= 
    \sum_{k\in\enne} \int_0^t S(t-s) Be^k\,dw_k(s)\\
    &= \sum_{k\in\enne} \int_0^t\!\!\int_\co G_{t-s}(\cdot,y)
    [Be^k](y)\, dy \,dw_k(s),
  \end{align*}
  where the series of ordinary It\^o integrals are no longer formal
  by virtue of (\ref{eq:hsc}).
  
  As far as the stochastic integral on the right-hand side of
  (\ref{eq:422}) is concerned, we notice (see e.g. \cite{Dalang-Quer})
  that it may be understood as a stochastic integral with respect to a
  cylindrical $Q$-Wiener process $\{ \bar{W}_h(t),\, h\in L^2,\, t\geq
  0\}$ in the sense of e.g. \cite{MetPell}. Namely, by definition, the
  latter is a centered Gaussian family of random variables such that,
  for any $h\in L^2$, the process $\{\bar{W}_h(t),\, t\geq 0\}$
  is a Brownian motion with variance $t\ip{Qh}{h}$ and,
  for all $s,t\geq 0$ and $h,g\in L^2$,
  \begin{equation}
   \E(\bar{W}_h(s)\bar{W}_g(t))=(s\land t)\langle Qh,g\rangle_2. 
  \label{eq:424}
  \end{equation}
  By an innocuous abuse of notation, this cylindrical $Q$-Wiener
  process will be also denoted by $\bar{W}$.  Let us introduce the
  Hilbert space $L^2_Q$, which we define as the completion of $L^2$
  with respect to $\ip{h}{g}_{L^2_Q}:=\ip{Qh}{g}_2$.  Then one can
  define the (real-valued) stochastic integral with respect to
  $\bar{W}$ of any $L^2_Q$-valued square integrable process as
  follows: let $\{\bar{e}^k\}_{k\in\enne}$ be a basis of
  $L^2_Q$ and $X \in L^2(\Omega \times [0,T] \to L^2_Q)$ (see
  e.g. \cite[Sec.~2]{Dalang-Quer}), and set
  \[
  \int_0^T\!\!\int_\co X(s,y) \,\bar{W}(ds,dy) := 
  \sum_{k\in\enne} \int_0^t \ip{X(s)}{\bar{e}^k}_{L^2_Q}\,d\bar{W}_{\bar{e}_k}(s).
  \]
  Moreover, the $L^2$-valued cylindrical Wiener process $W$ of
  \eqref{eq:423} determines a cylindrical $Q$-Wiener process $\bar{W}$,
  with $Q=BB^*$, as follows: for any $t\geq 0$ and $h \in L^2$, set
  \[
  \bar{W}_h(t) :=
  \sum_{k\in\enne} \int_0^t \ip{Be^k}{h}_2 \,dw_k(s).
  \]
  It is immediate that $\bar{W}_h(t)$ is a centered Gaussian random
  variable and
  \[
  \E \Big( \sum_k \int_0^t \ip{Be^k}{h}_2 \,dw_k(s) \Big)^2 
  = t \sum_k \ip{e^k}{B^*h}_2^2 = t \|B^*h\|^2_2 
  = t \ip{BB^*h}{h}_2 = t \ip{Qh}{h}_2.
  \]
  In a completely similar way one verifies the covariance condition
  (\ref{eq:424}).

  In order to prove (\ref{eq:422}), one just needs to take
  $\bar{e}^k:=Q^{-1/2} e^k$, where $Q^{-1/2}$ denotes the
  pseudo-inverse of $Q^{1/2}$, whence one easily verifies that
  $\bar{W}_{\bar{e}^k}(s)=w_k(s)$ and
  \[
  \int_\co G_{t-s}(\cdot,y)[B e^k](y)\,dy = \int_\co
  G_{t-s}(\cdot,y)[Q \bar{e}^k](y)\,dy.
  \]
  We have thus proved the identities
  \[
  \int_0^t S(t-s)B\,dW(s) = \sum_k \int_0^t\!\!\int_\co
  G_{t-s}(\cdot,y) [Be^k](y)\,dw_k(s) = \int_0^t\!\!\int_\co
  G_{t-s}(\cdot,y)\,\bar{W}(ds,dy).
  \]
  Finally, the estimate (\ref{unif-sup}) is an 
  immediate consequence of Theorem \ref{thm:wp}. 
\end{proof}

\section{Auxiliary results}
\label{sec:aux}
We collect here some tools that we shall need in the next section.  In
particular, we give a version of the chain rule for the Malliavin
derivative, where a (Malliavin) differentiable random variable is
composed with an increasing function of polynomial growth. The result,
also without the monotonicity assumption, is certainly well-known to
experts (cf. e.g. \cite[p.36]{Malliavin}), but we include a proof for
completeness and because most standard references contain only a proof
for functions of class $C^1_b$
(cf. e.g. \cite[Prop. 1.2.3]{nualart}). Moreover, we prove some
estimates for evolution operators generated by time-dependent bounded
perturbations of the Laplacian and for regularizations via Yosida
approximations as well as via mollification.

\smallskip

As usual, a function $\phi:\erre\to\erre$ is said to be of polynomial
growth if there exists $p \in \enne$ such that $|\phi(x)| \lesssim
1+|x|^p$ for all $x\in\erre$. We shall denote by $\Cpol^k(\erre)$ the
space of functions $\phi \in C^k(\erre)$ such that $\phi$,
$\phi',\ldots,\phi^{(k)}$ are of polynomial growth. It is not
difficult to see that, by the fundamental theorem of calculus, one can
equivalently say that $\Cpol^k(\erre)$ is the space of functions $\phi
\in C^k(\erre)$ such that $\phi^{(k)}$ is of polynomial growth, and
also that it is the space of functions $\phi \in C^k(\erre)$ for which
there exists $p \in \enne$ such that
\[
|\phi(x)| + |\phi'(x)| + \cdots + |\phi^{(k)}(x)| \lesssim 1 + |x|^p.
\]

\subsection{Malliavin calculus}
We shall repeatedly use the following strong-weak closability property
of the Malliavin derivative (cf. e.g. \cite[p.~78]{nualart-LNM}). We
use standard notation and terminology (see e.g. \cite{nualart}).
\begin{lemma}     \label{lm:clos}
  Let $k\in\enne$, $p \in ]1,\infty[$. Assume that $\lim_{n\to\infty}
  X_n=X$ in $L^p(\Omega)$ and $\sup_{n\in\enne}
  \|X_n\|_{\D^{k,p}}<\infty$. Then $X \in \D^{k,p}$ and $X_n \to X$
  weakly in $\D^{k,p}$.
\end{lemma}

We shall denote the Gaussian Hilbert space ``supporting'' the
Malliavin calculus by $H$. Moreover, we shall say that $X \in
\D^{k,\infty}$ if $X \in \D^{k,p}$ for all $p>1$.
\begin{lemma}     \label{lm:chain}
  Let $X \in \D^{1,\infty}$, $f\in \Cpol^1(\erre)$ and increasing. Then
  $f(X) \in \D^{1,\infty}$ and $Df(X)=f'(X) DX$.
\end{lemma}
\begin{proof}
  Assume, without loss of generality, that $|f(x)|+|f'(x)| \lesssim
  1+|x|^p$, with $p \in \enne$. Let $f_\lambda$, $\lambda>0$, denote
  the Yosida approximation of $f$. Since $f_\lambda \to f$ pointwise
  and $f$ is continuous, one has $f_\lambda(X) \to f(X)$ $\P$-a.s. as
  $\lambda \to 0$. Moreover, recalling that $|f_\lambda| \leq |f|$,
  hence that, for any $q \geq 1$,
  \[
  |f_\lambda(x)-f(x)|^q \leq 2^q|f(x)|^q \lesssim 1 + |x|^{qp},
  \]
  and $\E|X|^{qp}<\infty$, the dominated convergence theorem yields
  \begin{equation}         \label{eq:322}
    \lim_{\lambda \to 0} \E|f_\lambda(X)-f(X)|^q = 0
    \qquad \forall q \geq 1.
  \end{equation}
  Appealing to the inverse function theorem, it is easy to show that
  $f \in C^1$ implies $f'_\lambda \in C_b$, and
  \[
  J'_\lambda(x) = \frac{1}{1 + \lambda f'(J_\lambda(x))}, \qquad
  f'_\lambda(x) = \frac{f'(J_\lambda(x))}{1+\lambda f'(J_\lambda(x))}.
  \]
  In particular, since $f_\lambda$ is Lipschitz continuous, one has
  $f_\lambda \in C^1_b$, so that the ``classical'' chain rule
  (see e.g. \cite[Prop.~1.2.3]{nualart}) implies that $f_\lambda(X) \in
  \D^{1,\infty}$ and $Df_\lambda(X) = f_\lambda'(X) DX$.
  Let us check now that, for any $q\geq 1$, it holds
  \begin{equation}     \label{eq:334}
    \sup_{\lambda>0} \E \|D f_\lambda(X)\|_H^q < \infty.
  \end{equation}
  By the above expression for $f'_\lambda$ and $f' \geq 0$ it
  immediately follows that
  \[
  |f'_\lambda(x)| \leq |f'(J_\lambda(x))| \lesssim
  1 + |J_\lambda(x)|^p \leq 1 + |x|^p,
  \]
  hence, by Cauchy-Schwarz' inequality,
  \[
  \E \|D f_\lambda(X)\|_H^q \leq \big( \E |f_\lambda'(X)|^{2q}
  \big)^{1/2} \big( \E \|D X\|_H^{2q} \big)^{1/2}
  \lesssim \big( 1 + \E|X|^{2pq} \big)^{1/2}
  \big( \E \|DX\|_H^{2q} \big)^{1/2}.
  \]
  The right-hand side is finite and independent of $\lambda$, thus
  (\ref{eq:334}) is proved. In particular, (\ref{eq:322}) and
  (\ref{eq:334}) imply that $f(X) \in \D^{1,\infty}$ and
  $Df_\lambda(X)$ converges to $Df(X)$ weakly in $\LL^q(H)$ as $\lambda
  \to 0$, for all $q\geq 1$. Since the weak limit is unique, in order
  to prove that $Df(X)=f'(X) DX$, it suffices to show that
  \[
  \lim_{\lambda \to 0} \E \|Df_\lambda(X) - f'(X) DX\|_H^q = 0.
  \]
  Observe that
  \begin{align*}
    \E \|Df_\lambda(X) - f'(X) DX\|_H^q &=
    \E \big( |f'_\lambda(X)-f'(X)|^q \|DX\|_H^q \big) \\
    & \leq \big( \E |f'_\lambda(X)-f'(X)|^{2q} \big)^{1/2}
      \big(\E\|DX\|_H^{2q} \big)^{1/2},
  \end{align*}
  and that, by the expression of $f'_\lambda$, one has $f'_\lambda \to
  f'$, hence $f'_\lambda(X) \to f'(X)$ $\P$-a.s. as $\lambda \to
  0$. Moreover, taking into account that
  \[
  |f'_\lambda(x)-f'(x)| = \Big| \frac{f'(J_\lambda(x))}{1+\lambda
    f'(J_\lambda(x))} - f'(x) \Big| \leq |f'(J_\lambda(x))| + |f'(x)|
  \lesssim 1 + |x|^p,
  \]
  one also has
  \[
  |f'_\lambda(X)-f'(X)|^{2q} \lesssim 1 + |X|^{2pq},
  \]
  hence, by the dominated convergence theorem, $f'_\lambda(X) \to
  f'(X)$ in $\LL^{2q}$ as $\lambda \to 0$, and the proof is finished.
\end{proof}

\subsection{Time-dependent evolution operators}
Using the notation introduced in Section \ref{sec:wpa}, let
$F:[0,T] \to L_+^\infty$, and consider the following linear evolution
equation on $L^2$:
\begin{equation}     \label{eq:gnrc}
  dy(t) - \Delta y(t) + F(t)y(t) = 0, \qquad y(s)=y_0 \in L^2,
  \qquad 0 \leq s \leq t \leq T.
\end{equation}
Here $L^p_+$, $p\in[1,\infty]$, denotes the set of (equivalence
classes of) functions $\phi \in L^p$ such that $\phi \geq 0$
a.e.. The evolution operator $U(t,s)$ is then defined by
$y(t)=:U(t,s)y_0$.
\begin{prop}     \label{prop:eto}
  For any $0 \leq s < t \leq T$, the following properties hold
  true:
  \begin{itemize}
  \item[(i)] $U(t,s)$ is positivity preserving, i.e. $y_0 \geq 0$
    implies $U(t,s)y_0 \geq 0$;
  \item[(ii)] $U(t,s) \leq S(t-s)$, i.e. $y_0 \geq 0$ implies
    $U(t,s)y_0 \leq S(t-s)y_0$;
  \item[(iii)] $U(t,s)$ is ultracontractive, i.e. its $L^1 \to L^\infty$
    norm is finite.
  \item[(iv)] $U(t,s)$ is a kernel operator, i.e. there exists a
    function $k:[0,T]^2 \times \co^2 \rightarrow \erre_+$ such that
    \[
    \big[U(t,s)\phi\big](x) = \int_\co k(t,s;x,y)\phi(y)\,dy.
    \]
  \end{itemize}
\end{prop}
\begin{proof}
  (i) Let $y_0 \geq 0$ and $y(\cdot)$ be a strong solution (without loss
  of generality) of (\ref{eq:gnrc}), where we assume, for simplicity,
  $s=0$. Taking the scalar product with $y^-$ and integrating with
  respect to time, we obtain, denoting the $L^2$ norm by $\|\cdot\|$,
  \[
  \frac 12 \|y^-(t)\|^2 + \int_0^t \|\nabla y^-(r)\|^2\,dr 
  + \int_0^t \bip{F(r)y^-(r)}{y^-(r)}\,dr = \|y_0^-\|^2,
  \]
  hence $\|y^-(t)\|^2 \leq 2\, \|y_0^-\|^2 = 0$, i.e. $y(t) \geq 0$ for
  all $t \in [0,T]$.
  (ii) Let $z$ be the solution to $z'-\Delta z=0$, $z(0)=y_0$. Then one has
  \[
  \frac{d}{dt}(y-z)(t) - \Delta (y-z)(t) + F(t)y(t) = 0, \qquad
  (y-z)(0)=0,
  \]
  hence, taking the scalar product with $(y-z)^+$ and integrating,
  \[
  \big\| (y(t)-z(t))^+ \big\|^2 + \int_0^t \big\| \nabla(y(r)-z(r))^+
  \big\|^2\,dr + \int_0^t \bip{F(r)y(r)}{(y(r)-z(r))^+}\,ds = 0,
  \]
  which yields, recalling that, by (i), $y(r) \geq 0$ for all $r$,
  $y(t) \leq z(t)$ for all $t$. (iii) Note that, by (i), one has
  \begin{align*}
    |U(t,s)y_0| = |U(t,s)y_0^+ - U(t,s)y_0^-| &\leq
    |U(t,s)y_0^+| + |U(t,s)y_0^-|\\
    &= U(t,s)y_0^+ + U(t,s)y_0^- = U(t,s)|y_0|,
  \end{align*}
  therefore
  \[
  |U(t,s)y_0| \leq U(t,s)|y_0| \leq S(t-s)|y_0| \in L^\infty,
  \]
  which immediately implies $U(t,s)y_0 \in L^\infty$. (iv) is a direct
  consequence of (iii), thanks to a classical criterion of Dunford and
  Pettis (see e.g. \cite[Ch.~XI,~\S1]{Kantorovich}).
\end{proof}

\subsection{Regularizations}
In the next Lemma we use the notation introduced immediately after
Theorem \ref{thm:wp}.
\begin{lemma}     \label{lm:Yop}
  Let $f \in \Cpol^m(\erre)$ be an increasing function such that
  $|f^{(n)}(x)| \lesssim 1 + |x|^p$ for all $n=0,1,\ldots,m$. Then the
  following properties hold:
  \begin{itemize}
  \item[(i)] $f_\lambda$ and $J_\lambda$ belong to $C^m(\erre)$;
  \item[(ii)] for $\lambda \leq 1$, there exists $q \in \enne$,
    independent of $\lambda$, such that $|f^{(n)}_\lambda(x)| \lesssim
    1 + |x|^q$ for all $n=0,1,\ldots,m$.
  \item[(iii)] $f^{(n)}_\lambda$ converges pointwise to $f^{(n)}$ as
    $\lambda \to 0$.
  \end{itemize}
\end{lemma}
\begin{proof}
  Since $J_\lambda=(I+\lambda f)^{-1}$, and $f' \geq 0$, the inverse
  function theorem implies that, if $f$ is of class $C^m$, then also
  $J_\lambda$ is of class $C^m$. Moreover, the identity $\lambda
  f_\lambda = I-J_\lambda$ implies that also $f_\lambda$ is of class
  $C^m$, hence (i) is proved.
\smallskip\par\noindent
  (ii) The polynomial growth of $f_\lambda$ is obvious by the inequality
  $|f_\lambda| \leq |f|$. Moreover, recalling that $f_\lambda=f \circ
  J_\lambda$, that $J_\lambda$ is of class $C^m$ and is a contraction,
  we have
  \[
  |f'_\lambda(x)| = |f'(J_\lambda(x))|\,|J'_\lambda(x)| 
  \leq |f'(J_\lambda(x))|
  \lesssim 1+|J_\lambda(x)|^p \leq 1+|x|^p.
  \]
  Taking into account that $\lambda f''_\lambda=J''_\lambda$ and that
  \[
  f''_\lambda = f''(J_\lambda)(J'_\lambda)^2 + f'(J_\lambda)J''_\lambda,
  \]
  one gets $(1+\lambda f'(J_\lambda)) f''_\lambda =
  f''(J_\lambda)(J'_\lambda)^2$, which implies
  \begin{equation}     \label{eq:iter2}
  |f''_\lambda(x)| \leq |f''(J_\lambda(x))| 
  \lesssim 1+|J_\lambda(x)|^p \leq 1+|x|^p.
  \end{equation}
  Unfortunately it does not seem possible to extend such elementary
  arguments to obtain the polynomial growth of $f^{(n)}_\lambda$. We
  can nonetheless argue as follows: by Arbogast's
  formula\footnote{This formula is better known as Fa\`a di Bruno's
    formula. The latter attribution, however, seems to be historically
    incorrect.}  (see e.g. \cite{Gould}) one has
  \[
  f^{(n)}_\lambda = \sum \frac{n!}{b_1!b_2! \cdots b_n!} f^{(k)}(J_\lambda)
  \Big( \frac{J'_\lambda}{1!} \Big)^{b_1}
  \Big( \frac{J''_\lambda}{2!} \Big)^{b_2} \cdots
  \Big( \frac{J^{(n)}}{n!} \Big)^{b_n},
  \]
  where the sum is taken over all possible combinations of
  $\{b_1,b_2,\ldots,b_n\} \subset \enne \cup \{0\}$ such that
  \[
  b_1+2b_2+\cdots+nb_n=n
  \qquad \text{and} \qquad
  b_1+b_2+\cdots+b_n=k.
  \]
  In particular, note that there is only one possible term of the
  series containing $f'(J_k)$, precisely the one corresponding to
  $b_1=b_2=\cdots=b_{n-1}=0$, $b_n=1$, that is
  $f'(J_\lambda)J^{(n)}_\lambda$. Similarly, there is only one
  possible term containing $f^{(n)}(J_\lambda)$, precisely the one
  corresponding to $b_1=n$, $b_2=b_3=\cdots=b_n=0$, that is
  $f^{(n)}(J_\lambda)(J'_\lambda)^n$. Recalling that $\lambda
  f^{(n)}_\lambda = -J^{(n)}_\lambda$, we have
  \begin{equation}     \label{eq:arbo}
  (1+\lambda f'(J_\lambda)) f^{(n)}_\lambda =
  f^{(n)}(J_\lambda)(J'_\lambda)^n + S_{n-1},
  \end{equation}
  hence
  \[
  \big| f^{(n)}_\lambda \big| \leq 
  \big| f^{(n)}(J_\lambda) \big| + |S_{n-1}|,
  \]
  where $S_{n-1}$ is a finite sum of terms involving only
  $f^{(k)}(J_\lambda)$ and powers of $J^{(k)}_\lambda$, for
  $k=1,\ldots,n-1$. Using once again that $\lambda
  f^{(k)}_\lambda=-J^{(k)}_\lambda$, we conclude that $S_{n-1}$ is a
  finite sum of terms involving only $f^{(k)}(J_\lambda)$ and powers
  of $\lambda f^{(k)}_\lambda$, for $k=1,\ldots,n-1$. Taking $\lambda
  \leq 1$, since $|f^{(k)}(J_\lambda(x))| \lesssim 1+|x|^p$ for all
  $k=1,\ldots,n$, recalling that $f''_\lambda$ satisfies
  \eqref{eq:iter2}, we obtain that there exists $q_3 \in \enne$ such
  that $|f^{(3)}_\lambda(x)| \lesssim 1+|x|^{q_3}$. By iteration one
  ends up with $|f^{(k)}_\lambda(x)| \lesssim 1+|x|^{q_k}$ for each
  $k=0,1,\ldots,m$. Since $m$ is finite, this implies the claim.
%
%\smallskip\par\noindent
%
  (iii) It is known that $f_\lambda \to f$ pointwise as $\lambda \to
  0$. That $f'_\lambda \to f'$ pointwise as $\lambda \to 0$ has been
  proved in Lemma \ref{lm:chain}. Passing to the limit (pointwise) as
  $\lambda \to 0$ in \eqref{eq:arbo} we obtain
  \[
  \lim_{\lambda \to 0} f^{(n)}_\lambda = f^{(n)} 
  + \lim_{\lambda \to 0} S_{n-1},
  \]
  where we have used that $J_\lambda \to I$ and
  \[
  \lim_{\lambda \to 0} J'_\lambda(x) = \lim_{\lambda \to 0}
  \frac{1}{1+\lambda f'(J_\lambda(x))} = 1.
  \]
  The claim is proved if we show that $S_{n-1} \to 0$ pointwise as
  $\lambda \to 0$. For $n=2$ one has $S_1=0$, hence the claim
  holds. For $n \geq 3$, we observe that each term in $S_{n-1}$ is of
  the form
  \[
  c f^{(i)}(J_\lambda) \big(J'_\lambda\big)^{h_1}
  \big(J''_\lambda\big)^{h_2} \cdots
  \big(J^{(n-1)}_\lambda\big)^{h_{n-1}},
  \]
  where $c$ is a positive number, $1 \leq i,h_1 \leq n-1$, and
  $h_2,h_3,\ldots,h_{n-1}$ are nonnegative integers with at least one
  of them, say $h_s$, greater than 1. Let $2 \leq \sigma:=h_s$. Recall that
  \[
    J^{(\sigma)}_\lambda = -\lambda
  f^{(\sigma)}_\lambda,
  \]
  therefore the generic term of $S_{n-1}$, hence $S_{n-1}$ itself,
  converges to zero as $\lambda \to 0$.
\end{proof}

Let us introduce mollifiers in the following (standard) way: for
$\zeta \in C^\infty(\erre)$ positive, with support contained in
$[-1,1]$ and $\int_\erre \zeta=1$, set, for any $\beta>0$,
$\zeta_\beta(x):=\beta^{-1}\zeta(x/\beta)$.
\begin{lemma}     \label{lm:0}
  Let $f:\erre \to \erre$ be such that $|f(x)| \lesssim 1+|x|^p$, and
  $f_\beta=f\ast\zeta_\beta$, $\beta \leq 1$.  Then one has
  \[
  |f_\beta(x)| \leq N(1+|x|^p) \qquad \forall x \in \erre,
  \]
  where the constant $N$ does not depend on $\beta$.
\end{lemma}

\begin{proof}
  Assume that $|f(x)| \leq N_1(1+|x|^p)$. By the triangle inequality
  and the properties of $\zeta$ one has
  \begin{align*}
    |f_\beta(x)| &\leq \int_\erre |f(x-y)|\zeta_\beta(y)\,dy
    \leq N_1\int_\erre \big( 1 + |x-y|^p \big)\zeta_\beta(y)\,dy\\
    &\leq N_1(1+|x|^p) + N_1\int_\erre |y|^p\zeta_\beta(y)\,dy\\
    &= N_1(1+|x|^p) + N_1\beta^p \int_\erre |y|^p\zeta(y)\,dy\\
    &\leq N_1(2+|x|^p).
    \qedhere
  \end{align*}
\end{proof}

%-----------------------------------------------------------------------

\section{Existence and smoothness of the density}
\label{sec:exis-smooth}
In this section, we provide sufficient conditions on the data of the
problem implying that the law of $u(t,x)$ is absolutely continuous
with respect to the Lebesgue measure and that the corresponding
density is a differentiable function. More precisely, we will prove
the following result.
\begin{thm}     \label{thm:main} 
  Assume that $u_0 \in C(\overline{\co})$ and let $u \in
  \bigcap_{q\in\enne} \C_q$ be the mild solution to
  \eqref{eq:30}. Furthermore, assume that there exists $\gamma \in
  (0,2)$ such that, for all $x \in \co$, there exists a constant
  $c_x>0$ such that, for any $t \in (0,1)$,
  \begin{equation} \label{lower-s}
    g(x,t):= \int_{0}^{t} \|G_s(x,\cdot)\|^2_{L^2_Q} \, ds \geq c_x \, t^\gamma.
  \end{equation}
  \begin{itemize}
  \item[(a)] If $f\in \Cpol^1(\erre)$ is increasing, then for any
    $(t,x)\in ]0,T] \times \co$, the random variable $u(t,x)$ is
    absolutely continuous with respect to the Lebesgue measure.
  \item[(b)] Moreover, if for some integer $m \geq 2$, $f\in
    \Cpol^m(\erre)$, then the density of $u(t,x)$ belongs to
    $C^{m-2}(\erre)$.
  \end{itemize}
\end{thm}
As we shall see below, the assumption (\ref{lower-s}) is not needed to
prove Malliavin regularity of $u(t,x)$. It is instead needed to prove
finiteness of negative moments of the Malliavin matrix (which reduces
to a real random variable in the present setting). Moreover, as
explained in Remark \ref{rmk:existence}, in order to prove the
existence of density, condition (\ref{lower-s}) can be slightly
weakened. Nevertheless, for the sake of conciseness and clarity, we
have decided to state only one condition of the term $g(x,t)$.

\begin{rmk}
  The hypothesis on the initial datum in the previous theorem may be
  relaxed to $u_0 \in \LL^q(C(\overline{\co}))$ for all $q \geq 1$,
  and $u_0 \in L^\infty(\co \to \D^{1,\infty})$. However, it does not
  seem natural to assume the initial datum to have such regularity.
\end{rmk}

Let us give some examples of domains $\co$ and covariance operators
$Q=BB^*$ satisfying condition (\ref{lower-s}) above.

\begin{example}
  Let $d=1$, $\co=(0,1)$ and $B= \mbox{Id}$. Then (\ref{lower-s})
  holds with $\gamma=\frac12$ (see e.g. \cite[(A.3)]{Bally-Pardoux}).
\end{example}
  
\begin{example}
  Let $\co=(0,\pi)^d$. Define, for any $k=(k_1,\dots,k_d)\in
  \mathbb{N}^d$,
  \[
  e_k(x):=\left(\frac 2\pi\right)^{\frac d2} \sin(k_1 x_1)\cdots
  \sin(k_d x_d), \qquad x\in \co.
  \]
  Then, it is readily checked that the family $\{e_k\}_{k\in
    \mathbb{N}^d}$ is an orthonormal basis of $L^2(\co)$ such that
  \[
  -\Delta e_k = |k|^2 \, e_k,
  \]
  where $|k|^2:=k_1^2+\cdots+k_d^2$. Set $B=(I-\Delta)^{-m}$
  for $m\geq 0$. Then, since $Q=(I-\Delta)^{-2m}$, one has
  \begin{align*}
    g(x,t) & = \int_0^t \!\! \int_{\co} G_s(x,y) [Q G_s(x,\cdot)](y)\,dy\,ds\\
    & = \int_0^t \sum_{k\in \mathbb{N}^d} (1+|k|^2)^{-m}
        \ip{G_s(x,\cdot)}{e_k}_{L^2}^2 \,ds \\
    & = \int_0^t \sum_{k\in \mathbb{N}^d} (1+|k|^2)^{-m} \, e^{-2s |k|^2} |e_k(x)|^2 \,ds \\
    & = \frac12 \sum_{k\in \mathbb{N}^d} (1+|k|^2)^{-m} \, |k|^{-2}
    \,(1-e^{-2 t |k|^2}) |e_k(x)|^2.
  \end{align*}
  Using the fact that $|e_k(x)|$ is uniformly bounded with respect to
  $k$ and $x$, one easily verifies that $m>\frac d2 -1$ implies that
  the latter series is finite.  Moreover, we have that
  \[
  1-e^{-2t |k|^2} \geq \frac{2 t |k|^2}{1+2 t |k|^2} \geq \frac{2 t
    |k|^2}{1+2 T |k|^2}.
  \]
  Hence
  \[
  g(x,t)\geq t \; \sum_{k\in \mathbb{N}^d} (1+|k|^2)^{-m} \, (1+2 T
  |k|^2)^{-1} \, |e_k(x)|^2
  \]
  and this series can be bounded from below by any of its summands,
  such as the corresponding to $k=(1,\dots,1)\in
  \mathbb{N}^d$. Therefore,
  \[
  g(x,t) \geq c_x\, t, \quad \text{with} \quad c_x:= (1+d)^{-m} \,
  (1+2 T d)^{-1} \, \left(\frac2\pi\right)^{\frac d2} \sin(x_1)\cdots
  \sin(x_d).
  \]
  Since $x\in (0,\pi)^d$, it is clear that $c_x>0$ and this implies
  that condition (\ref{lower-s}) is fulfilled with $\gamma=1$.
\end{example}

Before turning to the study of the Malliavin differentiability of the
solution to our equation, let us recall that the underlying Gaussian
space on which to perform Malliavin calculus is given by the isonormal
Gaussian process on the Hilbert space $H:=L^2([0,T] \to L^2_Q)$, which
can be naturally associated to the cylindrical Wiener process
$\bar{W}$ with covariance $Q$. With a slight (but harmless) abuse of
notation we shall write $W$ instead of $\bar{W}$ for notational
convenience.

\subsection{Malliavin differentiability} 
\label{sec:malliavin}
The purpose of this subsection is to prove regularity of the
collection of random variables $\{u(t,x)\}_{(t,x)\in\co_T}$ in the
sense of Malliavin. Proposition \ref{prop:one} concerns the
Malliavin differentiability of order one, while Proposition
\ref{prop:smoothness} treats higher-order Malliavin derivatives. As
already mentioned, the two results rely on hypotheses (a) and (b) of
Theorem \ref{thm:main}, respectively, but not on the lower bound for
the stochastic convolution.
\begin{prop}     \label{prop:one}
  Assume that
  \begin{itemize}
  \item[(i)] $u_0 \in C(\overline{\co})$;
  \item[(ii)] $f$ is increasing and belongs to $\Cpol^1(\erre)$;
  \item[(iii)] $Q$ is positivity preserving.
  \end{itemize}
  Let $u \in \bigcap_{q \geq 1}\C_q$ be the mild solution to
  \eqref{eq:30}. Then, for any $(t,x)\in \co_T$, one has $u(t,x) \in
  \D^{1,\infty}$.  Moreover, the Malliavin derivative $Du(t,x)$
  satisfies the following linear equation in $H$:
  \begin{equation}
    Du(t,x)= v_0(t,x) + \int_0^t\!\!\int_\co G_{t-s}(x,y)
    (\eta-f'(u(s,y))) D u(s,y) \,dy\,ds
    \label{eq:120}
  \end{equation}
  where $v_0(t,x):=(\tau,z) \mapsto G_{t-\tau}(x,z)\,1_{[0,t]}(\tau)$,
  and it holds that, for all $q \geq 1$,
  \begin{equation} \label{sup:deriv}
    \sup_{(t,x) \in \co_T} \E\|Du(t,x)\|_H^q < \infty.
  \end{equation}
\end{prop}
\begin{proof}
  Since $f_\lambda$ is Lipschitz continuous and of class $C^1$, slight
  modifications of the ``classical'' results
  (cf. e.g. \cite{nualart-LNM}) imply that, for any $(t,x) \in \co_T$,
  $u_\lambda(t,x)$ belongs to $\mathbb{D}^{1,\infty}$ and satisfies
  the following linear deterministic integral equation with random
  coefficients:
  \[
    Du_{\lambda}(t,x)= v_0(t,x) + \int_0^t\!\!\int_{\co} G_{t-s}(x,y)
    (\eta-f'_{\lambda}(u_{\lambda}(s,y))) D u_{\lambda}(s,y) \,dy\,ds.
  \]
  Recall that, by Proposition~\ref{prop:1}, one has, for any $q \geq 1$,
  \[
  \E| u_\lambda(t,x) - u(t,x)|^q \to 0
  \]
  as $\lambda \to 0$, for all $(t,x) \in \mathcal{O}_T$. Therefore, by
  Lemma \ref{lm:clos}, in order to conclude that
  $u(t,x)$ belongs to $\D^{1,\infty}$ for all $(t,x) \in
  \mathcal{O}_T$, it suffices to show that for all $q\geq
  1$ and $(t,x)\in \co_T$, one has
  \begin{equation}      \label{eq:sup}
    \sup_{\lambda>0} \E \|D u_\lambda(t,x)\|^q_{H} <\infty.
  \end{equation}
  Let $\{h^k\}_{k\in\enne}$ be an orthonormal basis of $H$, and set
  \[
  \varphi_\lambda^k(t,x) := \ip{Du_{\lambda}(t,x)}{h^k}_H,
  \qquad (t,x) \in \co_T.
  \]
  Then $\varphi_{\lambda}^k(t):=\varphi_{\lambda}^k(t,\cdot)$, $0 \leq
  t \leq T$, satisfies the deterministic evolution equation with
  random coefficients
  \begin{equation}     \label{eq:251}
    \frac{d}{dt}\varphi_\lambda^k(t) - \Delta\varphi_\lambda^k(t)
    + F_\lambda(t) \varphi_\lambda^k(t) = \Phi^k(t),
    \qquad \varphi_{\lambda}^k(0)=0,
  \end{equation}
  where $F_\lambda(t):=f_\lambda'(u_\lambda (t,\cdot))-\eta$ and
  $\Phi^k(t):=Q h^k(t)$.  In fact, one has
  \begin{align*}
  \ip{v_0(t,x)}{h^k}_H &= 
  \int_0^t \ip{G_{t-s}(x,\cdot)}{h^k(s)}_{L^2_Q}\,ds =
  \int_0^t\!\!\int_\co G_{t-s}(x,y)[Qh^k(s)](y)\,dy\,ds\\
  &= \Big[\int_0^t S(t-s)Qh^k(s)\,ds\Big](x).
  \end{align*}
  From now we assume, without loss of generality, that $\eta=0$ (if
  not, it is enough to write the corresponding equation for $t \mapsto
  e^{-\eta t} \varphi_\lambda(t)$, multiplying by $e^{-\eta t}$ the
  term $\Phi^k(t)$).

  Fix $\omega \in \Omega$, and let $(s,t) \mapsto U_\lambda(t,s)$, $s
  \leq t$, denote the family of evolution operators generated by the
  time-dependent linear operator $\Delta - F_\lambda(t)$. Then we can
  write
  \[
  \varphi^k_\lambda(t) = \int_0^t U_\lambda(t,s)\Phi^k(s)\,ds, 
  \]
  thus also, denoting the integral kernel of $U_\lambda(t,s)$ (that
  exists by Proposition \ref{prop:eto}(iv)), by
  $k_\lambda(t,s;\cdot,\cdot)$,
  \[
  \varphi^k_\lambda(t,x) = \int_0^t\!\!\int_\co
  k_\lambda(t,s;x,y)\Phi^k(s,y)\,dy\,ds.
  \]
  This yields
  \begin{align*}
    \|Du_\lambda(t,x)\|_H^2 &= \sum_{k\in\enne} |\varphi_\lambda^k(t,x)|^2 =
    \sum_{k\in\enne} \Big| \int_0^t\!\!\int_\co
    k_\lambda(t,s;x,y)\Phi^k(s,y)\,dy\,ds \Big|^2\\
    &= \sum_{k\in\enne} \big|
       \ip{k_\lambda(t,\cdot;x,\cdot)1_{[0,t]}(\cdot)}
          {h^k}_{L^2([0,T]\to L^2_Q)} \big|^2\\
    &= \int_0^t \| k_\lambda(t,s;x,\cdot) \|^2_{L^2_Q}\,ds.
  \end{align*}
  Note that Proposition \ref{prop:eto}(ii) implies
  $k_\lambda(t,s) \leq G_{t-s}$ pointwise for all $0 \leq s < t \leq
  T$. Using that $Q$ is positivity preserving, we are left with
  \begin{equation}\label{eq:345}
  \|Du_\lambda(t,x)\|_H^2 \leq
  \int_0^t \|G_{t-s}(x,\cdot)\|^2_{L^2_Q} ds = \| v_0(t,x) \|^2_H.
  \end{equation}
  Let us now show that $\| v_0(t,x) \|^2_H$ is uniformly bounded over
  $t$ and $x$. In fact, one has
  \begin{align*}
    \| v_0(t,x) \|^2_H &= \int_0^t \| G_{t-\tau}(x,\cdot)
    \|^2_{L^2_Q}\,d\tau
    = \int_0^t \| B^* G_{t-\tau}(x,\cdot) \|^2_{2}\,d\tau\\
    &= \int_0^t \sum_{k\in\enne}
    \ip{G_{t-\tau}(x,\cdot)}{Be^k}^2_{2}\,d\tau = \int_0^t \sum_{k
      \geq 1} \Big(
    \int_\co G_{t-\tau}(x,y)[Be^k](y)\,dy \Big)^2\,d\tau\\
    &= \sum_{k\in\enne} \int_0^t \big[S(t-\tau)Be^k\big](x)^2\,d\tau,
  \end{align*}
  where $\{e^k\}_{k\in\enne}$ is an orthonormal basis of $L^2$. The
  identities
  \[
  \E|W_A(t,x)|^2 = \E\Big| \sum_{k\in\enne} \int_0^t
  \big[S(t-s)Be^k\big](x)\,dw_k(s) \Big|^2 = \sum_{k\in\enne} \int_0^t
  \big[S(t-s)Be^k\big](x)^2\,ds,
  \]
  yield $\|v_0(t,x)\|_H^2 = \E|W_A(t,x)|^2$ for all $(t,x) \in
  \co_T$. Thanks to Hypothesis \ref{hyp:csc} we infer that there
  exists a constant $N=N(q)$, independent of $\lambda$, such that
  \[
  \sup_{(t,x) \in \co_T} \E\|Du_\lambda(t,x)\|_H^q < N.
  \]
  We have thus proved that $u(t,x) \in \mathbb{D}^{1,\infty}$ for all
  $(t,x) \in \co_T$. It is therefore lawful to apply the
  Malliavin derivative to the equation satisfied by $u$, obtaining
  \[
    Du(t,x)= v_0(t,x) - \int_0^t\!\!\int_\co G_{t-s}(x,y)
    Df(u(s,y)\,dy\,ds.
  \]
  Then, appealing to the chain rule proved in Lemma
  \ref{lm:chain}, we obtain that the Malliavin derivative
  $Du(t,x)$ satisfies equation (\ref{eq:120}).

  In order to conclude, we only have to show that estimate
  (\ref{sup:deriv}) holds true. The argument is essentially the same
  as above, hence it is only sketched. Still assuming $\eta=0$ without
  loss of generality, setting $F(t,x):=f'(u(t,x))$, one has that
  $\varphi^k(t,x):=\ip{Du(t,x)}{h^k}_H$ satisfies the linear
  deterministic evolution equation with random coefficients
  \[
  \frac{d}{dt}\varphi^k(t) - \Delta\varphi^k(t) + F(t) \varphi^k(t) =
  \Phi^k(t), \qquad \varphi^k(0)=0.
  \]
  Let $\Omega' \subset \Omega$ with $\P(\Omega')=1$ such that $(t,x)
  \mapsto u(t,x,\omega) \in C([0,T] \times \bar{\co})$ for all $\omega
  \in \Omega'$. Fix $\omega \in \Omega'$. Then $(t,x) \mapsto F(t,x)$
  is positive and continuous, hence bounded on the compact set $[0,T]
  \times \bar{\co}$. One can then construct the evolution operator
  associated to $\Delta-F$, and proceeding exactly as above one
  arrives at
  \[
  \sup_{(t,x) \in \co_T} \E\|Du(t,x)\|_H^q < \infty,
  \]
  so that the proof is complete.
\end{proof}

\begin{rmk}
  Condition (iii) in Proposition \ref{prop:one} above does not need to
  be considered an important restriction.  Indeed, this condition is
  satisfied in the spatially homogeneous counterpart when dealing with
  existence and smoothness of the density for stochastic heat and wave
  equations in $\erre^d$ (see e.g. \cite{Nualart-QuerPOTA}).
\end{rmk}

\begin{prop}    \label{prop:smoothness}
  Let $m \in \enne$, $m \geq 2$. Assume that
  \begin{itemize}
  \item[(i)] $u_0 \in C(\overline{\co})$;
  \item[(ii)] $f$ is increasing and belongs to $\Cpol^m(\erre)$;
  \item[(iii)] $Q$ is positivity preserving.
  \end{itemize}
  Let $u \in \cap_{q\in\enne} \C_q$ be the unique mild solution to
  \eqref{eq:30}.  Then, for any $(t,x)\in \co_T$, one has $u(t,x) \in
  \D^{m,\infty}$ and
  \[
  \sup_{(t,x)\in\co_T} \|u(t,x)\|_{\D^{m,q}} < \infty \qquad \forall q
  \geq 1.
  \]
\end{prop}
\begin{proof}
  Let $p \in \enne $ be such that $|f(x)+ |f'(x)| + \cdots +
  |f^{(m)}(x)| \lesssim 1+|x|^p$ for all $r \in \erre$. By Proposition
  \ref{prop:one}, we have that $u(t,x) \in \D^{1,\infty}$ for all
  $(t,x)\in\co_T$. Let us consider the regularized equation
  (\ref{eq:reg}): since $f_\lambda$ needs not have bounded derivatives
  of order two and higher, we cannot apply ``classical'' results
  (cf. e.g. \cite{nualart-LNM}) to deduce that, for any $(t,x) \in
  \co_T$, $u_\lambda(t,x)$ belongs to $\D^{m,\infty}$. For
  this reason, we introduce a further regularization: let
  $\{\zeta_\beta\}_\beta$ be a family of mollifiers as in Lemma
  \ref{lm:0} above. Note that
  $f^{(n)}_{\lambda\beta} = f'_\lambda \ast \zeta^{(n-1)}_\beta$ for
  all $n \geq 1$, hence $f_{\lambda\beta}$ has bounded derivatives of
  every order. Let $u_{\lambda \beta}$ be the unique mild solution in
  $\cap_{q\in\enne} \C_q$ to the equation
  \begin{equation}     \label{eq:21}
    du_{\lambda\beta}(t) - \Delta u_{\lambda\beta}(t)\,dt 
    + f_{\lambda\beta} (u_{\lambda\beta} (t))\,dt =
    \eta u_{\lambda\beta}(t)\,dt + B\,dW(t),
    \qquad u_{\lambda\beta}(0)=u_0.
  \end{equation}
  We split the rest of the proof in three steps: first we show that,
  for any $(t,x) \in \co_T$, one has $u_{\lambda\beta}(t,x) \to
  u(t,x)$ in $\LL^q$ as $\beta \to 0$. Then we obtain the uniform bound
    \begin{equation}     \label{eq:terna}
    \sup_{(t,x)\in\co_T} \E\big\|D^n u_{\lambda\beta}(t,x)\big\|_{H^{\otimes n}}
    < N,
    \end{equation}
    where $N$ is a constant independent of $\lambda$ and
    $\beta$. Finally we pass to the limit as $\beta \to 0$ and
    $\lambda \to 0$.
\smallskip\par\noindent
  \textsl{Step 1.} We assume again, without loss of generality, that
  $\eta=0$. It is easily seen that it holds
  \[
  u_{\lambda\beta}(t) - u_\lambda(t) = \int_0^t S(t-s)\big(
  f_\lambda(u_\lambda(s)) -
  f_{\lambda\beta}(u_{\lambda\beta}(s))\big)\,ds,
  \]
  hence, recalling that $S(t)$ is contracting in $L^\infty(\co)$ and
  denoting the norm of this space by $\|\cdot\|$,
  \[
  \big\|u_{\lambda\beta}(t) - u_\lambda(t)\big\| \leq \int_0^t \big\|
  f_\lambda(u_\lambda(s)) -
  f_{\lambda\beta}(u_{\lambda\beta}(s))\big\|\,ds.
  \]
  This yields, by the triangle inequality,
  \begin{align*}
  \big\|u_{\lambda\beta}(t) - u_\lambda(t)\big\| &\leq \int_0^t \big\|
  f_{\lambda\beta}(u_{\lambda\beta}(s)) - f_{\lambda\beta}(u_\lambda(s))
  \big\|\,ds + \int_0^t \big\| f_{\lambda\beta}(u_\lambda(s))
  - f_\lambda(u_\lambda(s))\big\|\,ds\\
  &\leq \frac{1}{\lambda} \int_0^t \big\|
        u_{\lambda\beta}(s) - u_\lambda(s) \big\|\,ds + I_\beta,
  \end{align*}
  where
  \[
  I_\beta := \int_0^T \big\| f_{\lambda\beta}(u_\lambda(s)) -
  f_\lambda(u_\lambda(s))\big\|\,ds.
  \]
  By Gronwall's inequality and some obvious manipulations, one arrives at
  \[
  \E\sup_{t \leq T} \big\|u_{\lambda\beta}(t) - u_\lambda(t)\big\|^q
  \leq e^{qT/\lambda} \E I^q_\beta.
  \]
  Let us show that $\E I^q_\beta \to 0$ as $\beta \to 0$: since
  $f_\lambda$ is continuous, $f_{\lambda\beta}$ converges uniformly to
  $f_\lambda$ as $\beta \to 0$. Therefore, as $u_\lambda(s) \in
  C(\overline{\co})$ $\P$-a.s., we also have that the integrand in the
  definition of $I_\beta$ converges to zero $\P$-a.s. as $\beta \to
  0$. Taking into account that
  \[
  \big\| f_{\lambda\beta}(u_\lambda(s)) - f_\lambda(u_\lambda(s))\big\|^q
  \lesssim_\lambda 1 + \|u_\lambda(s)\|^q,
  \]
  and that $\E\int_0^T \|u_\lambda(s)\|^q \, ds <\infty$, we get, by the
  dominated convergence theorem, that $\E I^q_\beta \to 0$, hence also
  \[
  \E\sup_{t \leq T} \big\| u_{\lambda\beta}(t) - u_\lambda(t)
  \big\|^q_{C(\overline{\co})} \to 0
  \]
  as $\beta \to 0$, for any $q \geq 1$.

\smallskip

  \noindent
  \textsl{Step 2.} For $n=1$ it is easily seen that (\ref{eq:terna})
  holds true, simply by the previous proposition, noting that
  $f'_{\lambda\beta}=f'_\lambda \ast \zeta_\beta \geq 0$. For the sake
  of clarity, let us explicitly show, in the case $n=2$, how estimate
  \eqref{eq:terna} is implied by the corresponding one with $n=1$. Then the
  general induction step will be clear. As before, we shall assume,
  without loss of generality, that $\eta=0$.
  Since, as already observed before, $f_{\lambda\beta}$ has bounded
  derivatives of every order, we infer that $u_{\lambda \beta}(t,x)\in
  \D^{2,\infty}$, the iterated Malliavin derivative $D^2
  u_{\lambda\beta}(t,x)$ takes values in $H^{\otimes 2}$ and satisfies
  \begin{align*}
    D^2 u_{\lambda \beta}(t,x) &+ 
    \int_0^t\!\!\int_\co G_{t-s}(x,y) f''_{\lambda \beta}(u_{\lambda \beta}(s,y))
    (Du_{\lambda \beta}(s,y))^{\otimes 2}\,dy\,ds \\
    &+ \int_0^t\!\!\int_\co G_{t-s}(x,y)
       f'_{\lambda\beta}(u_{\lambda \beta}(s,y))) D^2 u_{\lambda \beta}(s,y)
       \,dy\,ds = 0.
  \end{align*}
  Let $\{h^k\}_{k \in \enne}$ be an orthonormal basis of $H^{\otimes 2}$ and set
  \[
  \varphi_{\lambda \beta}^k(t,x) := 
  \ip{D^2 u_{\lambda \beta}(t,x)}{h^k}_{H^{\otimes 2}}, \qquad k \in \enne.
  \]
  Then $\varphi_{\lambda\beta}^k(t):=\varphi_{\lambda\beta}^k(t,\cdot)$
  satisfies the following linear deterministic evolution equation with
  random coefficients
  \begin{equation}        \label{eq:25}
    \frac{d}{dt}\varphi_{\lambda\beta}^k(t)
    - \Delta \varphi_{\lambda\beta}^k(t)
    + F_{\lambda\beta}(t)\varphi_{\lambda\beta}^k(t)
    = \Phi^k_{\lambda\beta}(t), \qquad \varphi_{\lambda\beta}^k(0)=0,
  \end{equation}
  where
  \[
  F_{\lambda\beta}(t) := f'_{\lambda\beta}(u_{\lambda\beta}(t,\cdot)),
  \qquad \Phi^k_{\lambda\beta}(t) :=
  f''_{\lambda\beta}(u_{\lambda\beta}(t,\cdot))
  \ip{(Du_{\lambda\beta}(t,\cdot))^{\otimes 2}}{h^k}_{H^{\otimes 2}}.
  \]
  Then we have
  \[
  \varphi_{\lambda\beta}^k(t) = \int_0^t U_{\lambda\beta}(t,s)
  \Phi^k_{\lambda\beta}(s)\, ds,
  \]
  hence also, denoting the kernel of $U_{\lambda\beta}(t,s)$ by
  $k(t,s;\cdot,\cdot)$,
  \begin{align*}
    |\varphi_{\lambda\beta}^k(t,x)|^2 &= \Big| \int_0^t\!\!\int_\co
    k_{\lambda\beta}(t,s;x,y) \Phi^k_{\lambda\beta}(s,y)\,dy\,ds \Big|^2\\
    &\leq \Big( \int_0^t\!\!\int_\co
    k_{\lambda\beta}(t,s;x,y) |\Phi^k_{\lambda\beta}(s,y)|\,dy\,ds \Big)^2\\
    &\leq \Big( \int_0^t\!\!\int_\co
    G_{t-s}(x,y) |\Phi^k_{\lambda\beta}(s,y)|\,dy\,ds \Big)^2\\
    &\lesssim_T \int_0^t\!\!\int_\co
    G_{t-s}(x,y) |\Phi^k_{\lambda\beta}(s,y)|^2 \,dy\,ds,
  \end{align*}
  where we have used the estimate $k_{\lambda\beta} \leq G$ in the
  first inequality, and Cauchy-Schwarz' inequality in the third
  inequality, recalling that $S(t)$ is contracting in $L^\infty$ (cf. (\ref{eq:101})).
 
  Summing over $k$, Tonelli's theorem yields
  \begin{align*}
    \|D^2u_{\lambda\beta}(t,x)\|_{H^{\otimes 2}}^2 &=
    \sum_{k\in\enne} |\varphi^k_{\lambda\beta}(t,x)|^2 
    \lesssim_T \int_0^t\!\!\int_\co G_{t-s}(x,y)
    \sum_{k\in\enne} |\Phi^k_{\lambda\beta}(s,y)|^2 \,dy\,ds\\ 
    &= \int_0^t\!\!\int_\co G_{t-s}(x,y)
       \|\Phi_{\lambda\beta}(s,y)\|^2_{H^{\otimes 2}} \,dy\,ds.
  \end{align*}
  where
  \[
  \Phi_{\lambda\beta}(t,x) :=
  f''_{\lambda\beta}(u_{\lambda\beta}(t,x))
  (Du_{\lambda\beta}(t,x))^{\otimes 2}.
  \]
  H\"older's inequality and Tonelli's theorem then imply
  \[
  \E\|D^2u_{\lambda\beta}(t,x)\|_{H^{\otimes 2}}^q \lesssim_T
  \int_0^t\!\!\int_\co G_{t-s}(x,y)
       \E\|\Phi_{\lambda\beta}(s,y)\|^q_{H^{\otimes 2}} \,dy\,ds,
  \]
  that is,
  \[
  \sup_{(t,x)\in\co_T} \E\|D^2u_{\lambda\beta}(t,x)\|_{H^{\otimes 2}}^q
  \lesssim_T
  \sup_{(t,x)\in\co_T} \E\|\Phi_{\lambda\beta}(t,x)\|^q_{H^{\otimes 2}}.
  \]
  Let us show that the right-hand side is finite: by Cauchy-Schwarz'
  inequality, we have
  \[
  \E \|\Phi_{\lambda\beta}(t,x)\|^q_{H^{\otimes 2}} \lesssim 
  \big( \E|f''_{\lambda\beta}(u_{\lambda\beta}(t,x))|^{2q}\big)^{1/2}
  \big( \E\|Du_{\lambda\beta}(t,x)\|^{4q}_H \big)^{1/2}.
  \]
  Assume, without loss of generality, $\lambda \leq 1$, $\beta \leq
  1$. Since $f''_{\lambda\beta}=f''_\lambda\ast\zeta_\beta$, and, by
  Lemma \ref{lm:Yop}, there exists $\sigma \in \enne$ such that
  $|f''_\lambda(x)| \lesssim 1+|x|^\sigma$, by Lemma \ref{lm:0} we
  also have $|f''_{\lambda\beta}(x)| \lesssim 1+|x|^\sigma$. Therefore
  \[
  \E \|\Phi_{\lambda \beta}(t,x)\|^q_{H^{\otimes 2}} \lesssim 
  \big( 1+\E|u_{\lambda \beta}(t,x)|^{2q\sigma} \big)^{1/2}
  \big( \E\|Du_{\lambda\beta}(t,x)\|^{4q}_H \big)^{1/2},
  \]
  where both terms on the right hand side are uniformly bounded over
  $t$, $x$, $\lambda$, and $\beta$ by results already proved; in fact, as for $u_\lambda$ itself, 
  the boundedness of the first term on the right-hand side above 
  follows from Proposition 6.2.2 in \cite{cerrai-libro}, since $f_{\lambda \beta}$ is also monotone. The
  claim is then verified for $n=2$.

  The general case is proved by induction in a completely similar
  way. In particular, assume that, given $3 \leq n < m$, one has the
  uniform bound
  \[
  \sup_{(t,x)\in\co_T} \|u_{\lambda\beta}(t,x)\|_{\D^{n-1,q}} < N,
  \]
  with $N$ independent of $\lambda$ and $\beta$.  Let
  $\{h^k\}_{k\in\enne}$ be an orthonormal basis of $H^{\otimes n}$,
  and set
  \[
  \varphi_{\lambda \beta}^k(t,x):= 
  \ip{D^n u_{\lambda\beta}(t,x)}{h^k}_{H^{\otimes n}}.
  \]
  Then $\varphi^k(t,\cdot)$ satisfies an equation of the form
  (\ref{eq:25}), where $\Phi^k$ is a sum of finitely many terms
  depending on $u_{\lambda\beta}$ and on its Malliavin derivatives of
  order not greater than $n-1$, whence
  \[
  \sup_{(t,x)\in\co_T} \E\|\Phi_{\lambda\beta}(t,x)\|^q_{H^{\otimes n}} < N,
  \]
  with $N$ a constant that does not depend on $\lambda$ nor on $\beta$.
  Moreover, by an argument completely analogous to one already used
  before, one shows that
  \[
  \sup_{(t,x)\in\co_T} \E\|D^n u_{\lambda\beta}(t,x)\|^q_{H^{\otimes n}}
  \lesssim 
  \sup_{(t,x)\in\co_T} \E\|\Phi_{\lambda\beta}(t,x))\|^q_{H^{\otimes n}}
  < N,
  \]
  where $N$ is the same constant of the previous inequality.
\smallskip\par\noindent

\textsl{Step 3.} Let $q>1$. By the previous steps and Lemma
\ref{lm:clos}, passing to the limit as $\beta \to 0$, we obtain
$u_\lambda(t,x) \in \D^{m,q}$ for all $(t,x) \in \co_T$, and also, by
lower semicontinuity of the norm with respect to weak convergence,
\[
\E \|D^n u_\lambda(t,x)\|^q_{H^{\otimes n}} \leq \liminf_{\beta \to 0}
\E\|D^n u_{\lambda\beta}(t,x)\|^q_{H^{\otimes n}},
\]
which implies, together with the last inequality,
\[
\sup_{(t,x)\in\co_T} \E \|D^n u_\lambda(t,x)\|^q_{H^{\otimes n}} < N.
\]
Recalling that, by Proposition \ref{prop:1}, $u_\lambda(t,x) \to
u(t,x)$ in $\LL^q$ as $\lambda \to 0$ for all $(t,x) \in \co_T$,
appealing again to Lemma \ref{lm:clos}, we arrive at $u(t,x) \in
\D^{m,q}$ for all $(t,x)$. Since $q$ is arbitrary, we conclude that
$u(t,x) \in \D^{m,\infty}$ for all $(t,x)$.
\end{proof}

\subsection{Analysis of the Malliavin matrix}
In this subsection, we shall use a standard method in order to prove
that the inverse of the Malliavin matrix has moments of all orders
(see e.g. \cite[Theorem 6.2]{Nualart-QuerPOTA}).
\begin{prop}     \label{prop:mm}
  Assume that the hypotheses of Theorem \textnormal{\ref{thm:main}(a)} are
  satisfied, as well as condition (\ref{lower-s}).  Then, for any $(t,x)\in ]0,T]\times \co$, one has
  \[
  \E \|Du(t,x)\|_H^{-q} <\infty \qquad \forall q \geq 1.
  \]
\end{prop}
\begin{proof}
  By \cite[Lemma 2.3.1]{nualart}, it suffices to prove that for any
  $q\geq 2$, there exists $\varepsilon_0(q)>0$ such that, for all
  $\varepsilon\leq \varepsilon_0$,
  \begin{equation}
    \label{eq:52}
    \P\left( \|Du(t,x)\|^2_H <\varepsilon\right) \lesssim \varepsilon^q. 
  \end{equation}

  Let $(t,x) \in ]0,T] \times \mathcal{O}$ be fixed.  Observe that we
  are assuming the same hypotheses as in Proposition
  \ref{prop:one}. Hence, $u(t,x)$ belongs to $\mathbb{D}^{1,
    \infty}$ and the Malliavin derivative $D u(t,x)$ satisfies
  equation (\ref{eq:120}).  Then, using this latter equation, we can
  infer that, for any $\delta>0$ sufficiently small,
  \begin{equation*}
    \|Du(t,x)\|^2_H = \int_0^t \|D_{\tau} u(t,x)\|_{L^2_Q}^2 d\tau 
    \geq \int_{t-\delta}^t \|D_{\tau} u(t,x)\|_{L^2_Q}^2 d\tau
    \geq \frac{1}{2} g(x,\delta)-I(t,x, \delta),
  \end{equation*}
  where $g(x,\delta)$ is as in assumption (\ref{lower-s}) and
  \begin{equation*} 
      I(t,x, \delta) :=\int_0^{\delta} \bigg\Vert \int_{t-\tau}^t
      \int_{\mathcal{O}} G_{t-s}(x,y) (\eta-f'(u(s,y))) D_{t-\tau}
      u(s,y) \, dyds \bigg\Vert^2_{L^2_Q} d\tau.
  \end{equation*}
  Hence, using Chebyshev's inequality, we have, for all $\varepsilon>0$,
  \begin{equation}
    \label{eq:51}
      \P\left( \|Du(t,x)\|^2_H <\varepsilon\right) 
      \leq \P\left\{ I(t,x;\delta) \geq \frac{g(x,\delta)}{2} -\varepsilon\right\} \leq \left( \frac{g(x,\delta)}{2} -\varepsilon\right)^{-p} \E
      |I(t,x,\delta)|^p.
  \end{equation}
  
  Let us now find an upper bound for the $p$-th moment of
  $I(t,x,\delta)$.  For this, we start by applying Minkowski and
  H\"older's inequalities, the latter with respect to the measure on
  $[t-\delta,t]\times \co$ given by $G_{t-s}(x,y)dy ds$, to obtain that
  \begin{align}
    & \E |I(t,x,\delta)|^p \leq \E \left( \int_{t-\delta}^t \int_\co
        G_{t-s}(x,y) |\eta-f'(u(s,y))| \|D_{t-\cdot} u(s,y)\|_{L^2([0,\delta];
          L^2_Q)}
        \, dyds\right)^{2p}  \nonumber \\
      & \qquad \leq \left(\int_{t-\delta}^t\int_\co G_{t-s}(x,y)\, dy ds\right)^{2p-1}  \nonumber \\
      & \qquad \qquad \times
      \int_{t-\delta}^t \int_\co G_{t-s}(x,y)\E ( (\eta+|f'(u(s,y))|)^{2p} \|D_{t-\cdot} u(s,y)\|^{2p}_{L^2([0,\delta]; L^2_Q)} )  \, dyds  \nonumber \\
      & \qquad \lesssim \delta^{2p-1} \int_{t-\delta}^{t} \int_\co
      G_{t-s}(x,y) \E ( (\eta+|f'(u(s,y))|)^{2p} \|D_{t-\cdot}
      u(s,y)\|^{2p}_{L^2([0,\delta]; L^2_Q)} ) \, dyds,
  \label{eq:50}  
  \end{align}
  where we have also used the estimate (\ref{eq:101}).  Thus, applying
  the Cauchy-Schwarz inequality and appealing to the polynomial growth
  condition on $f'$ (say $|f'(z)|\lesssim 1+|z|^r$ for all $z\in
  \erre$), the right hand side of (\ref{eq:50}) can be estimated, up
  to a positive constant, by
  \[
  \begin{split}
    &\delta^{2p-1} \sup_{(s,y)\in [t-\delta,t]\times \co} (\E \|D_{t-\cdot} u(s,y)\|^{4p}_{L^2([0,\delta]; L^2_Q)})^{1/2}  \\
    &\qquad \qquad \times \int_{t-\delta}^{t} \int_\co G_{t-s}(x,y)
    \left(1+ (\E |u(s,y)|^{4r p})^{1/2} \right) dyds.
  \end{split}
  \]
  At this point, we can appeal to (\ref{sup:deriv}) to get
  \[
  \sup_{(s,y)\in [t-\delta,t]\times \co} (\E \|D_{t-\cdot}
  u(s,y)\|^{4p}_{L^2([0,\delta]; L^2_Q)})^{1/2} \leq C(T).
  \]
  Taking into account again estimate (\ref{eq:101}), and the uniform bound (\ref{unif-sup}), we can infer that
  \[
  \E |I(t,x,\delta)|^p \lesssim_T\, \delta^{2p}.
  \]
  Plugging this estimate into (\ref{eq:51}), we obtain that
  \[
  \P\left( \|Du(t,x)\|^2_H <\varepsilon\right) \lesssim_T\,
  \left(\frac{g(x,\delta)}{2} -\varepsilon\right)^{-p} \delta^{2p}.
  \]
On the other hand, (\ref{lower-s}) yields $g(x, \delta) \geq c_x \delta^{\gamma}$.
Thus, if we choose $\delta=\delta(\varepsilon,x)$ sufficiently small in such a way that
  $c_x \delta^{\gamma}=4 \varepsilon$, we get 
  \[
  \P\left( \|Du(t,x)\|^2_H <\varepsilon\right) \lesssim_T
  \frac{\delta^{2p}}{\varepsilon^p} \lesssim_{x,T} \varepsilon^{p (\frac{2}{\gamma}-1)},
  \]
where $\frac{2}{\gamma}-1>0$ by hypothesis.
  Therefore, going back to (\ref{eq:52}), it suffices to take $p=\frac{q \gamma}{2-\gamma}$ and the proof is completed.
\end{proof}

\begin{rmk}\label{rmk:existence}
 We should point out that, in fact, in order to prove the existence of density (i.e. Theorem \ref{thm:main}(a)),
 condition (\ref{lower-s}) may be slightly weakened. Precisely, one needs to prove that $\|Du(t,x)\|_H>0$ $\P$-a.s.
 First, by (\ref{eq:345}) and 
 the lower semicontinuity of the norm with respect to weak convergence, one gets that 
 for any $\delta\in (0,1)$ and $q\geq 1$,
 \[
  \E \| D u(t,x) \|^q_{L^2([t-\delta,t]; L^2_Q)} \leq g(x,\delta)^{\frac q2}.
 \]
 Then, similarly as above (see also \cite[Thm. 5.2]{Nualart-QuerPOTA}), one proves that
 \begin{equation}\label{00}
  \E | I(t,x,\delta)| \lesssim  \delta \int_0^\delta \int_{\co} G_s(x,y) g(y,\delta) \; dy ds,
 \end{equation}
 and that, for any $n\geq 1$ and $\delta\in (0,1)$,
 \[
  \P(\|Du(t,x)\|_H^2< \frac1n ) \leq \left(\frac{g(x,\delta)}{2}-\frac1n\right)^{-1} \E |I(t,x,\delta)|.
 \]
 Taking limit as $n\rightarrow \infty$ and using (\ref{00}), we end up with
 \[
  \P(\|Du(t,x)\|_H^2 =0 ) \lesssim  g(x,\delta)^{-1} \delta \int_0^\delta \int_{\co} G_s(x,y) g(y,\delta) \; dy ds.
 \]
 In conclusion, $u(t,x)$ has a density provided that the following two conditions are satisfied:
 \begin{itemize}
  \item[(a)] For any $x\in \co$, $g(x,\delta)>0$ for all $\delta>0$,
  \item[(b)] For any $x\in \co$, it holds
  \[
   \lim_{\delta\rightarrow 0} \; \frac{\delta}{g(x,\delta)} \int_0^\delta \int_{\co} G_s(x,y) g(y,\delta) \; dy ds = 0.  
  \]
 \end{itemize}
\end{rmk}

\subsection{Proof of Theorem \ref{thm:main}}
It is now just a matter of putting pieces together. In particular, in
view of Bouleau-Hirsch criterion (see
e.g. \cite[Thm.~2.1.3]{nualart}), (a) follows by Propositions
\ref{prop:one} and \ref{prop:mm}. Similarly, (b) follows by
Propositions \ref{prop:smoothness} and \ref{prop:mm}, applying
a general criterion of the Malliavin calculus (see
e.g. \cite[Prop.~2.1.5]{nualart} or \cite[Thm.~4.1]{Malliavin}).

%--------------------------------------------------------------------+

\section*{Acknowledgments}
Part of the work for this paper was carried out while the
authors were visiting the Hausdorff Institute for Mathematics in
Bonn, whose hospitality and financial support are gratefully
acknowledged. L. Quer-Sardanyons is also supported by the grant MICINN-FEDER Ref. MTM2009-08869.

%\bibliographystyle{amsplain}
%\bibliography{ref3,lluis3}

\begin{thebibliography}{10}

\bibitem{Bally-Pardoux}
V.~Bally and {\'E}.~Pardoux, \emph{Malliavin calculus for white noise driven
  parabolic {SPDE}s.}, Potential Analysis \textbf{9} (1998), no.~1, 27--64.
  \MR{1644120 (99g:60095)}

\bibitem{Bmax}
H.~Br{\'e}zis, \emph{Op\'erateurs maximaux monotones et semi-groupes de
  contractions dans les espaces de {H}ilbert}, North-Holland Publishing Co.,
  Amsterdam, 1973. \MR{0348562 (50 \#1060)}

\bibitem{Cardon-weber}
C.~Cardon-Weber, \emph{Cahn-{H}illiard stochastic equation: existence of the
  solution and of its density}, Bernoulli \textbf{7} (2001), 777--816.
  \MR{2054575 (2005f:60136)}

\bibitem{cerrai-libro}
S.~Cerrai, \emph{Second order {PDE}'s in finite and infinite dimension},
  Lecture Notes in Mathematics, vol. 1762, Springer-Verlag, Berlin, 2001.
  \MR{2002j:35327}

\bibitem{DP-K}
G.~Da~Prato, \emph{Kolmogorov equations for stochastic {PDE}s}, Birkh\"auser
  Verlag, Basel, 2004. \MR{2111320 (2005m:60002)}

\bibitem{DZ92}
G.~Da~Prato and J.~Zabczyk, \emph{Stochastic equations in infinite dimensions},
  Cambridge University Press, Cambridge, 1992. \MR{1207136 (95g:60073)}

\bibitem{Dalang-Quer}
R.~Dalang and L.~Quer-Sardanyons, \emph{Stochastic integrals for spde's: a
  comparison}, Expo. Math. \textbf{29} (2011), 67--109. \MR{2785545}

\bibitem{Fournier}
N.~Fournier and J.~Printems, \emph{Absolute continuity for some one-dimensional
  processes}, Bernoulli \textbf{16} (2010), no.~2, 343--360. \MR{2668905
  (2011d:60176)}

\bibitem{Gould}
H.~W. Gould, \emph{The generalized chain rule of differentiation with
  historical notes}, Util. Math. \textbf{61} (2002), 97--106. \MR{1899320
  (2003a:26007)}

\bibitem{gyongy-pardoux2}
I.~Gy{\"o}ngy and {\'E}.~Pardoux, \emph{On the regularization effect of
  space-time white noise on quasi-linear parabolic partial differential
  equations}, Probab. Theory Related Fields \textbf{97} (1993), no.~1-2,
  211--229. \MR{1240724 (94j:60123)}

\bibitem{Kantorovich}
L.~V. Kantorovich and G.~P. Akilov, \emph{Functional analysis}, second ed.,
  Pergamon Press, Oxford, 1982, Translated from the Russian by Howard L.
  Silcock. \MR{664597 (83h:46002)}

\bibitem{Leon-Nualart-Pet}
J.~A. Le{\'o}n, D.~Nualart, and R.~Pettersson, \emph{The stochastic {B}urgers
  equation: finite moments and smoothness of the density}, Infin. Dimens. Anal.
  Quantum Probab. Relat. Top. \textbf{3} (2000), no.~3, 363--385. \MR{1811248
  (2001m:60146)}

\bibitem{Malliavin}
P.~Malliavin, \emph{Stochastic analysis}, Grundlehren der Mathematischen
  Wissenschaften [Fundamental Principles of Mathematical Sciences], vol. 313,
  Springer-Verlag, Berlin, 1997. \MR{1450093 (99b:60073)}

\bibitem{Marquez-Sarra}
D.~M{\'a}rquez-Carreras, M.~Mellouk, and M.~Sarr{\`a}, \emph{On stochastic
  partial differential equations with spatially correlated noise: smoothness of
  the law}, Stochastic Process. Appl. \textbf{93} (2001), no.~2, 269--284.

\bibitem{MetPell}
M.~M{\'e}tivier and J.~Pellaumail, \emph{Stochastic integration}, Academic
  Press, New York, 1980. \MR{578177 (82b:60060)}

\bibitem{Millet-Sanz-AP1999}
A.~Millet and M.~Sanz-Sol{\'e}, \emph{A stochastic wave equation in two space
  dimension: smoothness of the law}, Ann. Probab. \textbf{27} (1999), no.~2,
  803--844. \MR{1698971 (2001e:60130)}

\bibitem{nualart}
D.~Nualart, \emph{The {M}alliavin calculus and related topics}, second ed.,
  Springer-Verlag, Berlin, 2006. \MR{2200233 (2006j:60004)}

\bibitem{nualart-LNM}
\bysame, \emph{Application of {M}alliavin calculus to stochastic partial
  differential equations}, A minicourse on stochastic partial differential
  equations, Lecture Notes in Math., vol. 1962, Springer, Berlin, 2009,
  pp.~73--109. \MR{2508774 (2010g:60156)}

\bibitem{Nualart-QuerPOTA}
D.~Nualart and L.~Quer-Sardanyons, \emph{Existence and smoothness of the
  density for spatially homogeneous {SPDE}s}, Potential Anal. \textbf{27}
  (2007), no.~3, 281--299. \MR{2336301 (2009i:60107)}

\bibitem{pardoux-zhang}
{\'E}.~Pardoux and Tu~Sheng Zhang, \emph{Absolute continuity of the law of the
  solution of a parabolic {SPDE}}, J. Funct. Anal. \textbf{112} (1993), no.~2,
  447--458. \MR{1213146 (94k:60095)}

\bibitem{Sanz-book}
M.~Sanz-Sol{\'e}, \emph{Malliavin calculus}, Fundamental Sciences, EPFL Press,
  Lausanne, 2005, With applications to stochastic partial differential
  equations. \MR{2167213 (2006h:60005)}

\bibitem{Walsh}
J.~B. Walsh, \emph{An introduction to stochastic partial differential
  equations}, \'Ecole d'\'et\'e de probabilit\'es de Saint-Flour, XIV---1984,
  Lecture Notes in Math., vol. 1180, Springer, Berlin, 1986, pp.~265--439.
  \MR{0876085 (88a:60114)}

\bibitem{Zaidi-Nualart}
N.~L. Zaidi and D.~Nualart, \emph{Burgers equation driven by a space-time white
  noise: absolute continuity of the solution}, Stochastics Stochastics Rep.
  \textbf{66} (1999), no.~3-4, 273--292. \MR{1692868 (2000b:60160)}

\end{thebibliography}

\providecommand{\bysame}{\leavevmode\hbox to3em{\hrulefill}\thinspace}
\providecommand{\MR}{\relax\ifhmode\unskip\space\fi MR }
\providecommand{\MRhref}[2]{%
  \href{http://www.ams.org/mathscinet-getitem?mr=#1}{#2}
}
\providecommand{\href}[2]{#2}

\end{document}